\documentclass{amsart}
\usepackage[utf8]{inputenc}
\usepackage{amsmath, amssymb, amsthm, amsfonts, mathtools, array, xfrac, float, multirow, indentfirst, url, enumerate, enumitem, comment, afterpage, pifont, makecell, xcolor, hyperref, bbm}
\usepackage[margin=1.5in]{geometry}

\usepackage[backend=bibtex, sorting=none, bibstyle=alphabetic, citestyle=alphabetic, sorting=nyt,maxbibnames=99, giveninits=false, isbn=false, url=false, doi=false, eprint=false]{biblatex}
\usepackage{comment}
\usepackage{thmtools}
\usepackage{cleveref}
\allowdisplaybreaks

\newtheorem{theorem}{Theorem}
\newtheorem{lemma}[theorem]{Lemma}

\newtheorem{corollary}{Corollary}
\newtheorem{remark}{Remark}

\renewcommand{\Re}{\mathrm{Re}}
\renewcommand{\Im}{\mathrm{Im}}

\theoremstyle{definition}
\newcommand{\ov}{\overline}
\newcommand{\mc}{\mathcal}
\newcommand{\mr}{\mathrm}
\newcommand{\md}{\,\mathrm{d}}

\newcommand{\wh}{\widehat}
\newcommand{\wt}{\widetilde}

\newcommand{\supp}{\mathrm{supp}\:}

\makeatletter
\declaretheoremstyle
    [headformat={\NOTE}, 
    notebraces={}{}, 
    notefont=\bfseries, 
    preheadhook=\def\thmt@space{}, 
    numbered=no
    ]{namedtheorem}
\makeatother

\begin{document}

\author{Tianyu Zhao}

\title[Conditional estimates on the argument of Dirichlet $L$-functions]{Conditional estimates on the argument of Dirichlet $L$-functions with applications to low-lying zeros}

\address{
    Department of Mathematics, The Ohio State University, 231 West 18th
    Ave, Columbus, OH 43210, USA.
}
\email{zhao.3709@buckeyemail.osu.edu}


\begin{abstract}
    Under the generalized Riemann hypothesis, we use Beurling--Selberg extremal functions to bound the mean and mean square of the argument of Dirichlet $L$-functions to a large prime modulus $q$. As applications, we give alternative proofs of several results on low-lying zeros of $L(s,\chi)$ and obtain a new lower bound on the proportion of $L(s,\chi)$ modulo $q$ with zeros close to the central point $s=1/2$. In particular, we show conditionally that for any $\beta>1/4$, there exist a positive proportion of Dirichlet $L$-functions whose first zero has height less than $\beta$ times the average spacing between consecutive zeros.
\end{abstract}

\maketitle

\section{Introduction}

Let $L(s,\chi)$ be the Dirichlet $L$-function associated to a primitive Dirichlet character $\chi$ modulo an odd prime $q$. If $T$ does not coincide with the ordinate of any non-trivial zero of $L(s,\chi)$, define
\[
S(T,\chi):=\frac{1}{\pi}\mathrm{arg}\: L\left(\frac{1}{2}+iT,\chi\right)=-\frac{1}{\pi}\int_{1/2}^\infty \Im \frac{L'}{L}(\sigma+iT,\chi)\md \sigma,
\]
otherwise we set $S(T,\chi):=\lim_{\varepsilon\to 0}(S(T+\varepsilon,\chi)+S(T-\varepsilon,\chi))/2$. This paper concerns the mean and mean square of the quantity
\[
\wt{S}(T,\chi):=S(T,\chi)+S(T,\ov{\chi}),
\]
which appears in the Riemann--von Mangoldt formula as the principal error term (see, e.g., \cite[Corollary 14.6]{MV}):
\begin{equation}\label{zero counting v1}
    N(T,\chi)=\frac{T}{\pi}\log \frac{qT}{2\pi e}+\wt{S}(T,\chi)-\frac{\chi(-1)}{4}+O\left(\frac{1}{T+1}\right), \hspace{0.5cm} T>0. 
\end{equation}
Here $N(T,\chi)$ counts the number of zeros $\rho_\chi=\beta_\chi+i\gamma_\chi$ of $L(s,\chi)$ with $0<\beta_\chi<1$ and $-T\leq \gamma_\chi\leq T$ (any zero with ordinate exactly equal to $\pm T$ is counted with weight $1/2$). For the sake of brevity we shall use the notations
\[
\mathbb{E}[\wt{S}(T,\chi)]:=\frac{1}{q-2} \sum_{\substack{\chi \bmod q \\ \chi\neq \chi_0}} \wt{S}(T,\chi) \quad\text{and}\quad \mathbb{E}[\wt{S}(T,\chi)^2]:=\frac{1}{q-2} \sum_{\substack{\chi \bmod q \\ \chi\neq \chi_0}} \wt{S}(T,\chi)^2,
\]
where $\chi_0$ stands for the principal character modulo $q$. We shall always assume that $T>0$ and $q$ is large.

It follows from classical arguments that $S(T,\chi)=O(\log q(T+1))$. Assuming the generalized Riemann hypothesis (GRH), Selberg \cite[Theorem 6]{Sel} proved that
\[
S(T,\chi)=O\left(\frac{\log q(T+1)}{\log\log q(T+3)}\right)
\]
uniformly in $q$ and $T>0$. More recently, the use of Beurling--Selberg extremal functions allowed Goldston and Gonek \cite[Theorem 2]{GolGon} to show, under the Riemann hypothesis, that $|S(T)|\leq (\frac{1}{2}+o(1))\frac{\log T}{\log\log T}$ in the case of the Riemann zeta-function $\zeta(s)$, and upon standard modifications their method would also produce a bound on $S(T,\chi)$ with the same leading constant under GRH. Later, Carneiro, Chandee and Milinovich \cite[Theorem 2]{CCM2013} managed to sharpen Goldston and Gonek's bound by a factor of two by working with more complicated extremal functions. Subsequently, the same authors \cite{CCM2015} gave a much simpler proof for self-dual $L$-functions such as $\zeta(s)$ and quadratic Dirichlet $L$-functions by exploiting symmetry of their zeros.

Substantial cancellations occur when we average over all $\chi\bmod q$. For fixed $\varepsilon>0$ and any $T\leq q^{1/4-\varepsilon}$, Selberg \cite[Theorem 8]{Sel} established that
\begin{equation}\label{selberg mean S(t,chi)}
    \mathbb{E}[S(T,\chi)]=O(1).
\end{equation}
Since $\mathbb{E}[\wt{S}(T,\chi)]=2\mathbb{E}[S(T,\chi)]$, this bound together with \eqref{zero counting v1} implies that the average spacing between consecutive zeros with height $\leq 1$, say, is $\frac{2\pi}{\log q}$. Under GRH, the $O(1)$ in \eqref{selberg mean S(t,chi)} can be refined to $O(\frac{\log q(T+1)}{\log(q\log(T+3))})$, valid for all $q$ and $T>0$, due to an earlier result of Titchmarsh \cite{Titch}. In this note, we first sharpen Titchmarsh's bound by applying the method of Carneiro, Chandee and Milinovich in \cite{CCM2015}.

\begin{theorem}\label{theorem 1}
    Assuming GRH,
    \begin{align*}
        \left|\mathbb{E}[\wt{S}(T,\chi)]\right| \leq & \frac{1}{2}+\frac{\log(T+1)}{2\log (q\log(T+3))}+\mc{E}(q,T)
    \end{align*}
    where
    \[
    \mc{E}(q,T)\ll \frac{\log q+ \log(T+1)\log\log(q\log(T+3))}{(\log q)^2+(\log\log(T+3))^2}
    \]
    with the implied constant being uniform in $q$ and $T>0$. In particular, as $q\to \infty$, $|\mathbb{E}[\wt{S}(T,\chi)]|\leq 1/2+o(1)$ if $T\ll q^{o(1)}$.
\end{theorem}

\begin{remark}
    In our applications we are primarily interested in the case when $T$ is small (to be precise, $T\ll \frac{1}{\log q}$). In the opposite regime where $q$ is negligibly small compared to $T$, our bound has shape $\frac{1}{2}\frac{\log T}{\log \log T}+O(\frac{\log T \log\log\log T}{(\log\log T)^2})$, resembling the bound on $2S(T,\chi)$ for individual Dirichlet $L$-functions in the $T$-aspect (see \cite[Theorem 5]{CCM2015}).
\end{remark}

This immediately recovers a number of existing results for low-lying zeros.

\begin{corollary}\label{corollary: lowest zero}
    Let $\mr{ord}_{s=1/2}L(s,\chi)$ denote the multiplicity of the zero (if any) of $L(s,\chi)$ at $s=1/2$ and $|\gamma_{\chi,0}|=\min|\gamma_\chi|$ the height of the lowest non-trivial zero of $L(s,\chi)$. Assuming GRH, 
    \begin{align}
        \frac{1}{q-2} \sum_{\substack{\chi \bmod q \\ \chi\neq \chi_0}} \underset{s=1/2}{\mr{ord}} L(s,\chi)\leq& \frac{1}{2}+O\left(\frac{1}{\log q}\right)\label{nonvanishing 1/2},\\
        \min_{\substack{\chi\bmod q\\\chi\neq \chi_0}} \frac{|\gamma_{\chi,0}|\log q}{2\pi}\leq& \frac{1}{4}+O\left(\frac{1}{\log q}\right)\label{min gamma chi},\\
        \max_{\substack{\chi\bmod q\\\chi\neq \chi_0}} \frac{|\gamma_{\chi,0}|\log q}{2\pi}\geq& \frac{1}{4}+O\left(\frac{1}{\log q}\right) \label{max gamma chi}.
    \end{align}
\end{corollary}

\begin{proof}
    An application of the argument principle gives
    \begin{equation}\label{zero counting}     
        N(T,\chi)=\frac{T}{\pi}\log \frac{q}{\pi} + \wt{S}(T,\chi) + \frac{1}{2\pi} \int_{-T}^T \frac{\Gamma'}{\Gamma }\left(\frac{1}{4}+\frac{\delta_\chi}{2}+i\frac{u}{2}\right)\md u
    \end{equation}
    where $\delta_\chi=(1+\chi(-1))/2$. This is an exact formula as opposed to \eqref{zero counting v1}. The integral is plainly $O(T)$ if $T=O(1)$. Averaging both sides of \eqref{zero counting} over all non-principal characters modulo $q$, taking $T\to 0$ and applying Theorem~\ref{theorem 1}, we readily obtain \eqref{nonvanishing 1/2}. The other two assertions follow similarly by taking $\frac{T}{\pi}\log \frac{q}{\pi}=\frac{1}{2}+O(\frac{1}{\log q})$.
\end{proof} 

The first estimate \eqref{nonvanishing 1/2}, which implies conditionally that the proportion of non-vanishing at $s=1/2$ is at least $1/2-o(1)$ as $q\to \infty$, was first due to Murty \cite{Mur}. The second estimate \eqref{min gamma chi} asserts that the height of the lowest non-trivial zero of the family of $L(s,\chi)$ modulo $q$ is at most $1/4+o(1)$ times the average spacing. This was showed by Hughes and Rudnick \cite[Corollary 8.2]{HuRu} as a corollary of their well-known one-level density formula:
\begin{equation}\label{one-level density}
    \frac{1}{q-2}\sum_{\substack{\chi\bmod q \\ \chi\neq \chi_0}}\sum_{\rho_\chi}f\left(\frac{\log q}{2\pi}\frac{\rho_\chi-1/2}{i} \right)=\int_{-\infty}^\infty f(x)\md x+O\left(\frac{1}{\log q}\right)
\end{equation}
where $f$ is any suitable test function whose Fourier transform $\wh{f}(u):=\int_{-\infty}^\infty f(x)e^{-2\pi iux}\md x$ is compactly supported in $(-2,2)$. In fact all three estimates in Corollary~\ref{corollary: lowest zero} can be deduced from \eqref{one-level density}, but a different test function is required for each case in contrast to our current approach. We also remark that the $O(\frac{1}{\log q})$ error terms in these estimates are not best possible and have recently been refined slightly by the author in \cite[Theorems 4--6]{Zhao} (so that both of the $o(1)$ terms mentioned at the beginning of this paragraph can be removed).

Next we bound the mean square of $\wt{S}(T,\chi)$ in terms of the Beurling--Selberg functions $R^\pm(x)=R^\pm_{\Delta,T}(x)$, $\Delta>0$ (see \eqref{def Beurling-selberg} for the definition).

\begin{theorem}\label{theorem 2}
    Assuming GRH, 
    \begin{multline*}
        \mathbb{E}[\wt{S}(T,\chi)^2] \leq 2\left(\frac{\log q(T+1)}{2\pi\Delta}\right)^2 + 2\left\lceil \frac{e^{2\pi\Delta}}{q}\right\rceil \int_{\frac{\log 2}{2\pi}}^\Delta u \left(\wh{R^+}(u)^2+\wh{R^-}(u)^2\right)\md u\\
        +O\left(\left\lceil \frac{e^{2\pi\Delta}}{q}\right\rceil \min\left\{1,T^2+\Delta^{-2}\right\}+\frac{e^{2\pi\Delta}}{q\sqrt{\Delta}}+\frac{\log q(T+1)}{\Delta} \left(\frac{1}{\Delta}+\frac{e^{\pi\Delta}}{q}\right)\right)
    \end{multline*}
    where $\lceil x\rceil$ denotes the smallest integer greater than or equal to $x$. In particular, for $T\leq q^N$ where $N>0$ is fixed,
    \begin{equation}\label{second moment bound}
    \mathbb{E}[\wt{S}(T,\chi)^2] \leq \frac{2}{\pi^2}\min\{\log\log q,\log(T\log q+1)\}+O_N(1).
    \end{equation}
\end{theorem}

This should be compared to \cite[Theorem 9]{Sel} where Selberg obtained asymptotics of even moments of $S(T,\chi)$ without assuming GRH:
\[
\mathbb{E}[S(T,\chi)^{2n}]=\frac{(2n)!}{n!(2\pi)^{2n}}(\log\log q)^n+O\left((\log\log q)^{n-1/2}\right)
\]
where $n\in \mathbb{N}$ and $T\leq q^{1/4-\varepsilon}$. Taking $n=1$, this trivially yields an unconditional upper bound of $\frac{2}{\pi^2}\log\log q+O(\sqrt{\log\log q})$ for $\mathbb{E}[\wt{S}(T,\chi)^2]$ in the same range of $T$, which is almost as sharp as \eqref{second moment bound} for $T$ greater than any positive constant. However, observe that $\eqref{second moment bound}$ gives an $O(1)$ bound when $T\ll \frac{1}{\log q}$, which turns out crucial to establishing the existence of a positive proportion of $\chi\bmod q$ with low-lying zeros. In this direction, Selberg \cite[Theorem 3]{Sel} unconditionally proved a remarkable result that the proportion of $\chi\bmod q$ having a zero $1/2+i\gamma$ with $|\gamma|<\frac{2\pi\beta}{\log q}$ becomes positive and eventually converges to 1 as $\beta\to \infty$. (For this he used essentially the same method that allowed him to prove that a positive proportion of the non-trivial zeros of $\zeta(s)$ lie on the critical line.) More precisely, he showed that this proportion is at least $1-O(\beta^{-1/2})$ without providing an explicit expression in terms of $\beta$. Assuming GRH, Hughes and Rudnick \cite[Theorem 8.3]{HuRu} derived such an explicit lower bound using the variance of the one-level density:
\begin{multline}\label{Hughes Rudnick proportion}
    \liminf_{\substack{q\to \infty}} \frac{1}{q-2}\#\bigg\{\chi\neq \chi_0: \frac{|\gamma_{\chi,0}|\log q}{2\pi}<\beta\bigg\}
    \\
    \geq 1-\frac{3+\pi^2+72\beta^2-8\pi^2\beta^2+48\beta^4+16\pi^2\beta^4}{12\pi^2(4\beta^2-1)^2}.
\end{multline}
The right-hand side becomes positive for $\beta>0.633\ldots$ and converges to $0.891\ldots$ as $\beta\to \infty$, and so Selberg's unconditional result is in fact asymptotically stronger. This was improved to \cite[Theorem 7]{Zhao}
\begin{equation}\label{zhao proportion}
    \begin{cases}
        \dfrac{1}{1+\dfrac{3+\pi^2+72\beta^2-8\pi^2\beta^2+48\beta^4+16\pi^2\beta^4}{12\pi^2(4\beta^2-1)^2}}, & 1/2 <\beta<0.909\ldots,\\
        \dfrac{1}{1+\dfrac{0.193\ldots}{\beta^2}}, & \beta\geq 0.909\ldots,
    \end{cases}
\end{equation}
which indeed converges to 1 as $\beta\to \infty$, at a rate of $1-O(\beta^{-2})$. Moreover, it gives a positive proportion for $\beta>1/2$, compared to $0.633\ldots$. We now demonstrate that Theorem~\ref{theorem 2} allows us to further reduce $1/2$ down to $1/4$, which is the limit of currently available methods considering the $1/4$-barrier in \eqref{min gamma chi}.

\begin{corollary}\label{corollary: proportion}
    Assuming GRH, for all $\beta> 1/4$ we have 
    \begin{multline}\label{lower bound proportion}
        \liminf_{q\to \infty} \frac{1}{q-2}\#\bigg\{\chi\neq \chi_0:\frac{|\gamma_{\chi,0}|\log q}{2\pi}< \beta\bigg\}\\
        \geq \frac{4\beta^2-2\beta+1/4}{4\beta^2-2\beta+2+\displaystyle 2\int_{0}^1 u\left(\wh{R_{1,\beta}^{+}}(u)^2 +\wh{R_{1,\beta}^{-}}(u)^2\right)\md u}.
    \end{multline}  
\end{corollary}

One can numerically compute the integral above for any given $\beta$. When $\beta=1/2$, for instance, we find that more than $10\%$ of $\chi\bmod q$ has a zero within half of the average spacing when $q$ is large. For a rough estimate of this integral, we see from Lemma~\ref{lemma: Beurling selberg} (vi) that it is
\[
=\int_{0}^1 2u\frac{\sin^2(2\pi \beta u)}{(\pi u)^2}\md u + O(1) = \frac{2}{\pi^2}\int_{0}^\beta \frac{\sin^2(2\pi u)}{u}\md u + O(1) = \frac{\log (\beta+1)}{\pi^2}+O(1).
\]
Hence the lower bound provided by \eqref{lower bound proportion} grows like $1-O(\frac{\log \beta}{\beta^2})$ as $\beta\to \infty$, which is slightly weaker than \eqref{zhao proportion} asymptotically. In fact, numerical computation indicates that \eqref{zhao proportion} outperforms \eqref{lower bound proportion} quickly after $\beta$ passes $1/2$ ($\approx 0.55$ to be precise).

Prior results in \cite{HuRu} and \cite{Zhao} rely on the so-called ``positivity technique" for detecting first zeros, which involves applying the explicit formula to a certain test function $f$ that is non-negative outside some bounded interval whose length is proportional to $\beta$, but at the same time $\wh{f}(0)<0$ and the support of $\wh{f}$ is restricted in order for the off-diagonal terms to make a negligible contribution when we square the prime sum in the explicit formula. Unfortunately, these conditions cannot be simultaneously satisfied when $\beta$ is too small, so we could only obtain a positive proportion when $\beta>1/2$. More concretely, one may think of the requirements on $f$ to be $f(x)\geq 0$ for $|x|\geq \beta$, $\wh{f}(0)<0$ and $\supp(\wh{f})\subset [-1,1]$. Such an $f$ fails to exist if $\beta\leq 1/2$, according to the work of Carneiro, Ismoilov and Ramos on the sign uncertainty principle (see Theorem 2 and the subsequent Remark (iii) in \cite{CIR}). Therefore, any attempt to improve this would have to circumvent such limitation from Fourier analysis. Our approach here boils down to bounding the mean and mean square of $\wt{S}(T,\chi)$ where $T$ is proportional to $\beta$. To this end we seek approximations $f$ for the characteristic function of the interval $[-T,T]$ and then apply the explicit formula to $f$ as before. Even though the limitation from number theory still exists, namely $\supp(\wh{f})$ must be restricted, no other essential conditions need to be imposed on $f$, and thankfully such approximations are available no matter how short the interval is, i.e., how small $\beta$ is.

We also briefly remark that the unconditional analogues of Theorems~\ref{theorem 1} and \ref{theorem 2} with applications to low-lying zeros have been worked out by G. Hiary and the author in \cite{HiaZha}.

We end this section by giving a quick derivation of Corollary~\ref{corollary: proportion}.

\begin{proof}[Proof of Corollary~\ref{corollary: proportion}]
    Denote by $\mc{Q}_\beta$ the set of characters in question, which is non-empty by \eqref{min gamma chi} since we are assuming $\beta>1/4$. Equivalently, $\mc{Q}_\beta=\{\chi\neq \chi_0:N(\frac{2\pi\beta}{\log q},\chi)>0\}=\{\chi\neq \chi_0: \wt{S}(\frac{2\pi\beta}{\log q},\chi)+2\beta+O(\frac{1}{\log q})>0\}$. By the Cauchy--Schwarz inequality,
    \begin{align*}
        \# \mc{Q}_\beta \geq &\frac{\left[\sum_{\chi\in \mc{Q}_\beta}\left(\wt{S}(\frac{2\pi\beta}{\log q},\chi)+2\beta+O(\frac{1}{\log q})\right)\right]^2}{\sum_{\chi\in \mc{Q}_\beta}\left(\wt{S}(\frac{2\pi\beta}{\log q},\chi)+2\beta+O(\frac{1}{\log q})\right)^2} \geq \frac{\left[\sum_{\chi\neq \chi_0}\left(\wt{S}(\frac{2\pi\beta}{\log q},\chi)+2\beta+O(\frac{1}{\log q})\right)\right]^2}{\sum_{\chi\neq \chi_0}\left(\wt{S}(\frac{2\pi\beta}{\log q},\chi)+2\beta+O(\frac{1}{\log q})\right)^2}.
    \end{align*}
    Using Theorem~\ref{theorem 1} we obtain
    \begin{align*}
        \liminf_{q\to \infty} \frac{\#\mc{Q}_\beta}{q-2}\geq & \liminf_{q\to \infty}  \frac{(2\beta+\mathbb{E}[\wt{S}(\frac{2\pi\beta}{\log q},\chi)])^2}{4\beta^2+4\beta\mathbb{E}[\wt{S}(\frac{2\pi\beta}{\log q},\chi)]+\mathbb{E}[\wt{S}(\frac{2\pi\beta}{\log q},\chi)^2]} \\
        \geq & \liminf_{q\to \infty} \min_{|\lambda|\leq 1/2} \frac{(2\beta+\lambda)^2}{4\beta^2+4\beta\lambda+\mathbb{E}[\wt{S}(\frac{2\pi\beta}{\log q},\chi)^2]}.
    \end{align*}
    A quick calculus exercise shows that this ratio is minimized at $\lambda=-1/2$. The corollary now follows from an application of Theorem~\ref{theorem 2} with $\Delta=\frac{\log q}{2\pi}$ and $T=\frac{2\pi\beta}{\log q}$ along with the observation that
    \[
    \int_{\frac{\log 2}{2\pi}}^\Delta u\wh{R_{\Delta,T}^{\pm}}(u)^2 \md u = \int_{\frac{\log 2}{\log q}}^1 u\wh{R_{1,\beta}^{\pm}}(u)^2 \md u = \int_{0}^1 u\wh{R_{1,\beta}^{\pm}}(u)^2 \md u+O\left(\frac{1}{\log q}\right).
    \]
\end{proof}

\section{Preparations and lemmas}
For $s\in\mathbb{C}$ let 
\[
B^\pm(s):=\left(\frac{\sin \pi s}{\pi}\right)^2\left(\frac{2}{s}+\sum_{n=1}^\infty \frac{1}{(s-n)^2}-\sum_{n=1}^\infty \frac{1}{(s+n)^2}\right)\pm \left(\frac{\sin \pi s}{\pi s}\right)^2,
\]
and then put
\begin{equation}\label{def Beurling-selberg}
    R_{\Delta,T}^{\pm}(s):=\frac{B^{\pm}(\Delta(T+s))+B^{\pm}(\Delta(T-s))}{2}
\end{equation}
where $\Delta>0$. The even real entire functions $R^{\pm}(s)=R_{\Delta,T}^{\pm}(s)$ are called the Beurling--Selberg majorant and minorant for the interval $[-T,T]$, respectively, and have found  numerous applications in analytic number theory. We first record some of their essential properties in the following lemma. The reader is referred to \cite{Vaa} for a detailed discussion on such functions and proofs of assertions in the lemma.

\begin{lemma}\label{lemma: Beurling selberg}
    The functions $R^{\pm}(s)$ satisfy the following properties:
    \begin{enumerate}[label=(\roman*)]
        \item $R^-(x)\leq \mathbbm{1}_{[-T,T]}(x)\leq R^+(x)$ for all $x\in \mathbb{R}$, where $\mathbbm{1}_{-[T,T]}$ stands for the characteristic function of $[-T,T]$ (normalized at the endpoints);

        \item $\displaystyle \int_{-\infty}^\infty \left|R^{\pm}(x)-\mathbbm{1}_{[-T,T]}(x)\right| \md x=\Delta^{-1}$;

        \item $R^{\pm}(s)\ll e^{2\pi \Delta |\Im\: s|}$;

        \item $R^{\pm}(x)\ll \min(1,\Delta^{-2}(|x|-T)^{-2})$ for $|x|>T$;

        \item $\wh{R^{\pm}}(u)=0$ for $|u|\geq \Delta$;

        \item $\wh{R^{\pm}}(u)=\dfrac{\sin(2\pi Tu)}{\pi u}+O(\Delta^{-1})$ for $|u|<\Delta$; in particular, $\wh{R^{\pm}}(0)=2T \pm \Delta^{-1}$.
    \end{enumerate}
\end{lemma}

We shall also use Weil's explicit formula specialized to $L(s,\chi)$, which relates a sum over the non-trivial zeros of $L(s,\chi)$ to a sum over primes.

\begin{lemma}\label{lemma: explicit formula}
    Suppose that $f(s)$ is analytic in the strip $|\Im\:s|<1/2+\varepsilon$ for some $\varepsilon>0$ and that $f(s)\ll (1+|s|)^{-1-\delta}$ for some $\delta>0$ as $|\Re(s)|\to \infty$. Then
    \begin{multline}\label{explicit formula}
        \sum_{\rho_\chi} f\left(\frac{\rho_\chi-1/2}{i}\right)=\frac{\wh{f}(0)}{2\pi} \log \frac{q}{\pi} + \frac{1}{2\pi} \int_{-\infty}^\infty f(u) \Re \frac{\Gamma'}{\Gamma}\left(\frac{1}{4}+\frac{\delta_\chi}{2}+i\frac{u}{2}\right)\md u \\ 
        -\frac{1}{\pi} \sum_{n=1}^\infty \Re\left\{\chi(n)\wh{f}\left(\frac{\log n}{2\pi}\right)\right\}\frac{\Lambda(n)}{\sqrt{n}},
    \end{multline}
    where the sum on the left-hand side runs over all non-trivial zeros of $L(s,\chi)$, $\delta_\chi$ was defined in \eqref{zero counting}, and $\Lambda(n)$ denotes the von Mangoldt function, which takes on the value $\log p$ when $n=p^m$ is a prime power and 0 otherwise. 
\end{lemma}
\begin{proof}
    See, e.g., \cite[Theorem 5.12]{IwaKow}.
\end{proof}

Note that since we assume GRH, the left-hand side of \eqref{explicit formula} can be replaced by $\sum_{\gamma_\chi} f(\gamma_\chi)$.

The next two lemmas are concerned with the prime sum in \eqref{explicit formula} when applied to $R^{\pm}$. For our purposes, we shall assume that $\Delta$ is large, of size $\gg \log q$.

\begin{lemma}\label{lemma: prime sum}
    For any fixed $\varepsilon>0$,
    \begin{align*}
        \frac{1}{q-2} 
        \sum_{\substack{\chi \bmod q \\ \chi\neq \chi_0}}\sum_{n=1}^\infty\Re\{\chi(n)\}\wh{R^{\pm}}\left(\frac{\log n}{2\pi}\right)\frac{\Lambda(n)}{\sqrt{n}} \ll 
        \begin{cases}
            \dfrac{e^{\pi \Delta}}{q}, & \text{if}\:\:2\pi\Delta<(1+\varepsilon)\log q;\\
            \dfrac{e^{\pi \Delta}}{q\Delta}, &\text{if}\:\: 2\pi\Delta\geq (1+\varepsilon)\log q.
        \end{cases}
    \end{align*}
\end{lemma}

\begin{proof}
    By Lemma~\ref{lemma: Beurling selberg} (v) and the orthogonality relation of characters
    \begin{equation}\label{ortho}
        \sum_{\substack{\chi\bmod q\\\chi\neq \chi_0}}\chi(n)=
        \begin{cases}
            q-2 & \text{if $n\equiv 1 \bmod q$},\\
            0 & \text{if $n\equiv 0\bmod q$},\\
            -1 & \text{otherwise},
        \end{cases}
    \end{equation} 
    the sum over $n=p$ where $p$ is a prime equals
    \begin{align*}
        =& \sum_{\substack{p\leq e^{2\pi \Delta}\\ p\equiv 1\bmod q}}\wh{R^{\pm}}\left(\frac{\log p}{2\pi}\right)\frac{\log p}{\sqrt{p}} + \frac{1}{q-2}\sum_{p\leq e^{2\pi \Delta}}\wh{R^{\pm}}\left(\frac{\log p}{2\pi}\right)\frac{\log p}{\sqrt{p}}\\
        \ll& \sum_{\substack{p\leq e^{2\pi \Delta}\\ p\equiv 1\bmod q}}\frac{1}{\sqrt{p}} + \frac{1}{q}\sum_{p\leq e^{2\pi \Delta}}\frac{1}{\sqrt{p}}.
    \end{align*}
    In the last line we used the fact that $|u\wh{R^\pm}(u)|=O(1)$, which is implied by Lemma~\ref{lemma: Beurling selberg} (vi). The second term above is $\ll\frac{e^{\pi\Delta}}{q\Delta}$ by the prime number theorem. For the first sum, recall that by the Brun--Titchmarsh theorem \cite{MV2}, $\pi(x;q,1)\leq \frac{2x}{\phi(q)\log(x/q)}$ for $x>2q$ where $\phi(q)$ denotes Euler's totient function and $\pi(x;q,a)$ denotes the number of primes $p\leq x$ such that $p\equiv a\bmod q$. Therefore, if $2\pi\Delta\geq (1+\varepsilon)\log q$, then this sum is $\ll \frac{e^{\pi\Delta}}{q\Delta}$, and it is always $\ll \frac{e^{\pi\Delta}}{q}$ by the trivial bound $\pi(x;q,1)\leq x/q$. Consequently, we see that the contribution from primes satisfies the stated estimate. In a similar vein, the contribution from higher prime powers is
    \[
    \ll \sum_{m=2}^\infty \sum_{\substack{p\leq e^{\frac{2\pi\Delta}{m}} \\ p\equiv 1\bmod q}}\frac{1}{mp^{m/2}}+\frac{1}{q}\sum_{m=2}^\infty\sum_{p\leq e^{\frac{2\pi\Delta}{m}}} \frac{1}{mp^{m/2}}\ll \frac{\Delta}{q}+\frac{\log \Delta}{q}\ll \frac{\Delta}{q},
    \]
    which is subsumed by the main term.
\end{proof}

\begin{lemma}\label{lemma: square of prime sum}
    \begin{multline*}
        \frac{1}{q-2} 
        \sum_{\substack{\chi \bmod q \\ \chi\neq \chi_0}} \left(\sum_{n=1}^\infty\Re\{\chi(n)\}\wh{R^\pm}\left(\frac{\log n}{2\pi}\right)\frac{\Lambda(n)}{\sqrt{n}}\right)^2 \\
        \leq \left\lceil \frac{e^{2\pi\Delta}}{q}\right\rceil \left(2\pi^2 \int_{\frac{\log 2}{2\pi}}^\Delta u \wh{R^\pm}(u)^2\md u+O\left(\min\left\{1,T^2+\Delta^{-2}\right\}\right)\right)+O\left(\frac{e^{2\pi\Delta}}{q\sqrt{\Delta}}\right).
    \end{multline*}
\end{lemma}

\begin{proof}
    Writing 
    \[
    4\cdot \Re\{\chi(n_1)\}\Re\{\chi(n_2)\}=\chi(n_1)\chi(n_2)+\chi(n_1)\ov{\chi}(n_2)+\ov{\chi}(n_1)\chi(n_2)+\ov{\chi}(n_1)\ov{\chi}(n_2)
    \]
    and using orthogonality \eqref{ortho}, we can expand the square and bound the left-hand side by
    \begin{multline}\label{second moment two terms}
        \leq \frac{1}{2}\sum_{\substack{n_1n_2\equiv 1\bmod q\\ n_i\leq e^{2\pi\Delta}}}\wh{R^\pm}\left(\frac{\log n_1}{2\pi}\right)\wh{R^\pm}\left(\frac{\log n_2}{2\pi}\right) \frac{\Lambda(n_1)\Lambda(n_2)}{\sqrt{n_1n_2}}\\
        +\frac{1}{2}\sum_{\substack{n_1\equiv n_2\bmod q\\ n_i\leq e^{2\pi\Delta}}}\wh{R^\pm}\left(\frac{\log n_1}{2\pi}\right)\wh{R^\pm}\left(\frac{\log n_2}{2\pi}\right) \frac{\Lambda(n_1)\Lambda(n_2)}{\sqrt{n_1n_2}}\\
        +O\left(\frac{1}{q}\left(\sum_{p\leq e^{2\pi\Delta}}\wh{R^\pm}\left(\frac{\log p}{2\pi}\right) \frac{\log p}{\sqrt{p}}\right)^2\right).
    \end{multline}
    The error term is $O(\frac{e^{2\pi\Delta}}{q\Delta^2})$ again because $|u\wh{R^\pm}(u)|=O(1)$. The first term is
    \begin{align*}
        \ll \sum_{\substack{kq+1=p_1p_2 \\ p_i\leq e^{2\pi\Delta}}} \frac{1}{\sqrt{kq+1}}\ll \frac{1}{\sqrt{q}}\sqrt{\frac{e^{2\pi\Delta}}{\Delta}\frac{e^{2\pi\Delta}}{q}}=\frac{e^{2\pi\Delta}}{q\sqrt{\Delta}},
    \end{align*}
   where the second inequality is justified by the observation that there are $\ll \frac{e^{2\pi\Delta}}{\Delta}$ choices of $p_1$, and for each fixed $p_1$ there are at most $\frac{e^{2\pi\Delta}}{q}$ choices of $p_2$ such that $p_1p_2\equiv 1\bmod q$. Lastly, the second term in \eqref{second moment two terms} is
   \begin{align*}
        \leq& \frac{1}{4}\sum_{\substack{n_1\equiv n_2\bmod q\\ n_i\leq e^{2\pi\Delta}}} \left[\wh{R^\pm}\left(\frac{\log n_1}{2\pi}\right)^2\frac{\Lambda(n_1)^2}{n_1}+\wh{R^\pm}\left(\frac{\log n_2}{2\pi}\right)^2\frac{\Lambda(n_2)^2}{n_2}\right]\\
        \leq & \frac{1}{2} \left\lceil \frac{e^{2\pi\Delta}}{q}\right\rceil \left(\sum_{p\leq e^{2\pi\Delta}}\wh{R^\pm}\left(\frac{\log p}{2\pi}\right)^2\frac{(\log p)^2}{p} + O\left(\min\left\{1,T^2+\Delta^{-2}\right\}\right)\right)\\
        =&\left\lceil \frac{e^{2\pi\Delta}}{q}\right\rceil \left(2\pi^2 \int_{\frac{\log 2}{2\pi}}^\Delta u \wh{R^\pm}(u)^2\md u + O\left(\min\left\{1,T^2+\Delta^{-2}\right\}\right)\right),
    \end{align*}
    where we used $|\wh{R^{\pm}}(u)|\ll \min\{u^{-1},T+\Delta^{-1}\}$ to estimate the contribution from higher prime powers in the big-$O$ term. This proves the lemma.
\end{proof}

We now proceed to the proofs of our main theorems.

\section{Proofs of Theorems~\ref{theorem 1} and \ref{theorem 2}}

\subsection{Proof of Theorem~\ref{theorem 1}}
Writing  $N(T,\chi)=\sum_{\rho_\chi} \mathbbm{1}_{[-T,T]}(\gamma_\chi)$, we have, by \eqref{zero counting} and Lemma~\ref{lemma: Beurling selberg} (i),
\begin{align*}
    \mathbb{E}[\wt{S}(T,\chi)]
    =&\frac{1}{q-2}\sum_{\substack{\chi \bmod q \\ \chi\neq \chi_0}} N(T,\chi)-\frac{T}{\pi} \log\frac{q}{\pi} 
    \\
    &\hspace{3cm}-\frac{1}{2\pi(q-2)}\sum_{\substack{\chi \bmod q \\ \chi\neq \chi_0}} \int_{-T}^T \frac{\Gamma'}{\Gamma}\left(\frac{1}{4}+\frac{\delta_\chi}{2}+i\frac{u}{2}\right)\md u \\
    \leq & \frac{1}{q-2} \sum_{\substack{\chi \bmod q \\ \chi\neq \chi_0}}\sum_{\rho_\chi} R^+(\gamma_\chi)-\frac{T}{\pi} \log\frac{q}{\pi}\\
    &\hspace{3cm}-\frac{1}{2\pi(q-2)} \sum_{\substack{\chi \bmod q \\ \chi\neq \chi_0}} \int_{-T}^T \frac{\Gamma'}{\Gamma}\left(\frac{1}{4}+\frac{\delta_\chi}{2}+i\frac{u}{2}\right)\md u.
\end{align*}
After replacing the sum over zeros using the explicit formula \eqref{explicit formula}, we rewrite the above as
\begin{multline*}
    =\frac{\wh{R^+}(0)}{2\pi}\log \frac{q}{\pi}-\frac{T}{\pi}\log\frac{q}{\pi}-\frac{1}{\pi(q-2)} 
    \sum_{\substack{\chi \bmod q \\ \chi\neq \chi_0}}\sum_{n=1}^\infty\Re\{\chi(n)\}\wh{R^+}\left(\frac{\log n}{2\pi}\right)\frac{\Lambda(n)}{\sqrt{n}}\\
    +\frac{1}{2\pi(q-2)}\sum_{\substack{\chi \bmod q \\ \chi\neq \chi_0}}\int_{-\infty}^{\infty} \left(R^+(u)-\mathbbm{1}_{[-T,T]}(u)\right) \frac{\Gamma'}{\Gamma}\left(\frac{1}{4}+\frac{\delta_\chi}{2}+i\frac{u}{2}\right)\md u.
\end{multline*}
The first two terms add up to $\frac{1}{2\pi\Delta}\log \frac{q}{\pi}$. The prime sum can be handled by Lemma~\ref{lemma: prime sum}. To bound the integral we break it into two ranges and apply Stirling's formula 
\[
\Re\frac{\Gamma'}{\Gamma}\left(\frac{1}{4}+\frac{\delta_\chi}{2}+i\frac{u}{2}\right)=\log(|u|+1)+O(1)
\]
together with Lemma~\ref{lemma: Beurling selberg} (ii) and (iv). It follows that
\begin{align*}
    &\int_{|u|\leq 2T+1} \left(R^+(u)-\mathbbm{1}_{[-T,T]}(u)\right) \frac{\Gamma'}{\Gamma}\left(\frac{1}{4}+\frac{\delta_\chi}{2}+i\frac{u}{2}\right)\md u\leq \frac{\log(T+1)+O(1)}{\Delta},\\
    &\int_{|u|> 2T+1} \left(R^+(u)-\mathbbm{1}_{[-T,T]}(u)\right) \frac{\Gamma'}{\Gamma}\left(\frac{1}{4}+\frac{\delta_\chi}{2}+i\frac{u}{2}\right)\md u\ll \frac{\log(T+1)+O(1)}{\Delta^2(T+1)}.
\end{align*}
Combining these estimates, we arrive at
\begin{align*}
    \mathbb{E}[\wt{S}(T,\chi)] \leq &\frac{\log q(T+1)}{2\pi \Delta}+O\left(\frac{e^{\pi \Delta}}{q\Delta}+\frac{1}{\Delta}\right),
\end{align*}
provided $2\pi\Delta\geq (1+\varepsilon)\log q$ (see Lemma~\ref{lemma: prime sum}).
    Upon choosing $\pi\Delta=\log (q\log q(T+1))-\log\log (q\log q(T+1))$, the above expression is at most
\begin{align*}
    \leq &\frac{1}{2}+\frac{1}{2}\frac{\log(T+1)-\log \log q(T+1)+\log\log (q\log q(T+1))}{\log(q\log q(T+1))-\log\log (q\log q(T+1))}\\ &\hspace{6cm}+O\left(\frac{\log q(T+1)}{[\log(q\log q(T+1))]^2}\right)\\
    =&\frac{1}{2}+\frac{\log(T+1)}{2\log (q\log q(T+1))}+O\left(\frac{\log(T+1)\log\log (q\log q(T+1))+\log q(T+1)}{(\log q)^2+(\log\log q(T+1))^2}\right)\\
    =&\frac{1}{2}+\frac{\log(T+1)}{2\log (q\log(T+3))}+O\left(\frac{\log q+\log(T+1)\log\log (q\log (T+3))}{(\log q)^2+(\log\log (T+3))^2}\right).
\end{align*}

A similar argument involving $R^-$ yields the desired lower bound, thereby finishing the proof.

\subsection{Proof of Theorem~\ref{theorem 2}}  
Following the same lines of reasoning as in the proof of Theorem~\ref{theorem 1}, and using the simple fact that if $a\leq b\leq c$ where $a,b,c\in \mathbb{R}$, then $b^2\leq a^2+c^2$, we have $\mathbb{E}[\wt{S}(T,\chi)^2]\leq I^++I^-$ where
\[
I^{\pm}=\frac{1}{q-2}\sum_{\substack{\chi \bmod q \\ \chi\neq \chi_0}} \Bigg(\sum_{\rho_\chi} R^{\pm}(\gamma_\chi)-\frac{T}{\pi} \log\frac{q}{\pi}-\frac{1}{2\pi} \int_{-T}^T \frac{\Gamma'}{\Gamma}\left(\frac{1}{4}+\frac{\delta_\chi}{2}+i\frac{u}{2}\right)\md u \Bigg)^2.
\]
Lemmas~\ref{lemma: Beurling selberg}, \ref{lemma: explicit formula} and \ref{lemma: square of prime sum} together imply that
\begin{align*}
    I^{\pm}
    &= \frac{1}{q-2}\sum_{\substack{\chi \bmod q \\ \chi\neq \chi_0}} \Bigg(\pm\frac{1}{2\pi\Delta}\log\frac{q}{\pi}+\frac{1}{2\pi}\int_{-\infty}^{\infty} \left(R^{\pm}(u)-\mathbbm{1}_{[-T,T]}(u)\right) \frac{\Gamma'}{\Gamma}\left(\frac{1}{4}+\frac{\delta_\chi}{2}+i\frac{u}{2}\right)\md u \\
    &\hspace{5cm}-\frac{1}{\pi} \sum_{n=1}^\infty\Re\{\chi(n)\}\wh{R^{\pm}}\left(\frac{\log n}{2\pi}\right)\frac{\Lambda(n)}{\sqrt{n}}\Bigg)^2\\
    &\leq \left(\frac{\log q(T+1)}{2\pi\Delta}+O(\Delta^{-1})\right)^2+\left\lceil \frac{e^{2\pi\Delta}}{q}\right\rceil \left(2\int_{\frac{\log 2}{2\pi}}^\Delta u \wh{R^\pm}(u)^2\md u+O\left(\min\left\{1,T^2+\Delta^{-2}\right\}\right)\right)\\
    &\hspace{2cm}+O\left(\frac{e^{2\pi\Delta}}{q\sqrt{\Delta}}\right)+\mc{E}_2,
\end{align*}
where $\mc{E}_2$ corresponds to the cross term when we expand the square. It follows from Lemma~\ref{lemma: prime sum} and the definition of $\delta_\chi$ that
\begin{align*}
    \mc{E}_2=&O\left(\frac{\log q}{\Delta}\frac{e^{\pi\Delta}}{q}\right)-\frac{1}{\pi^2(q-2)}\left(\int_{-\infty}^{\infty} \left(R^{\pm}(u)-\mathbbm{1}_{[-T,T]}(u)\right) \frac{\Gamma'}{\Gamma}\left(\frac{3}{4}+i\frac{u}{2}\right)\md u\right)\\
    &\hspace{4cm}\times \sum_{n=1}^\infty\sum_{\substack{\text{even}\:\chi\bmod q\\\chi\neq\chi_0}} \Re\{\chi(n)\}\wh{R^{\pm}}\left(\frac{\log n}{2\pi}\right)\frac{\Lambda(n)}{\sqrt{n}}\\
    &\hspace{2.5cm}-\frac{1}{\pi^2(q-2)}\left(\int_{-\infty}^{\infty} \left(R^{\pm}(u)-\mathbbm{1}_{[-T,T]}(u)\right) \frac{\Gamma'}{\Gamma}\left(\frac{1}{4}+i\frac{u}{2}\right)\md u\right)\\
    &\hspace{4cm}\times \sum_{n=1}^\infty\sum_{\substack{\text{odd}\:\chi\bmod q\\\chi\neq\chi_0}} \Re\{\chi(n)\}\wh{R^{\pm}}\left(\frac{\log n}{2\pi}\right)\frac{\Lambda(n)}{\sqrt{n}}.
\end{align*}
Using orthogonality relations for even and odd characters and proceeding as in the proof of Lemma~\ref{lemma: prime sum}, we have
\[
\mc{E}_2=O\left(\frac{\log q(T+1)}{\Delta}\frac{e^{\pi\Delta}}{q}\right).
\]
This establishes the first claim in Theorem~\ref{theorem 2}. By Lemma~\ref{lemma: Beurling selberg}(vi), 
\begin{align*}
    \int_{\frac{\log 2}{2\pi}}^\Delta u \wh{R^\pm}(u)^2\md u=&\int_{\frac{\log 2}{2\pi}}^\Delta \frac{\sin^2 (2\pi Tu)}{\pi^2 u}\md u+O(1)\\
    =&\int_{\frac{\log 2}{2\pi}T}^{\Delta T}\frac{\sin^2(2\pi u)}{\pi^2 u}\md u+O(1)\\
    \leq & \frac{1}{2\pi^2}\min \{\log (\Delta+1), \log (\Delta T+1)\}+O(1).
\end{align*}
If $T\leq q^N$, we find by taking $2\pi\Delta=\log q$ that
\[
\mathbb{E}[\wt{S}(T,\chi)^2] \leq I^++I^-\leq \frac{2}{\pi^2}\min\{\log\log q,\log(T\log q+1)\}+O_N(1),
\]
which proves the second claim \eqref{second moment bound}.

\section{Further remarks}
The restriction of $q$ to prime moduli throughout this paper can be removed as done in \cite{Zhao}. Moreover, the analysis we have carried out can be adapted for studying zeros near an arbitrary point $1/2+iT_0$ on the critical line. Analogous to \eqref{zero counting}, the number of zeros in the interval $[T_0-h,T_0+h]$ (again normalized at the endpoints) equals
\[   
\frac{h}{\pi}\log \frac{q}{\pi} + \wt{S}(T_0, h,\chi) + \frac{1}{2\pi} \int_{T_0-h}^{T_0+h} \Re \frac{\Gamma'}{\Gamma }\left(\frac{1}{4}+\frac{\delta_\chi}{2}+i\frac{u}{2}\right)\md u.
\]
where 
\[
\wt{S}(T_0,h,\chi):=S(T_0+h,\chi)-S(T_0-h,\chi).
\]
The task is then to bound the mean and mean square of the quantity $\wt{S}(T_0,h,\chi)$ using the Beurling--Selberg majorant and minorant for the interval $[T_0-h,T_0+h]$, defined by (with a slight abuse of notation)
\[
R^\pm(s):=\frac{B^{\pm}(\Delta(h-T_0+s))+B^{\pm}(\Delta(h+T_0-s))}{2}
\]
(c.f. \eqref{def Beurling-selberg}). Note that $\wh{R^{\pm}}$ are not real since $R^{\pm}$ are no longer even. Now
\[
\wh{R^\pm}(u)=e^{-2\pi iT_0u}\frac{\sin(2\pi h u)}{\pi u}+O(\Delta^{-1}), \quad |u|\leq \Delta,
\]
and in particular $\wh{R^\pm}(0)=2h\pm \Delta^{-1}$. However, since the explicit formula \eqref{explicit formula} still apply in this case, so does our subsequent argument after some straightforward modifications. For instance, in Lemma~\ref{lemma: square of prime sum}, we need to replace $\Re\{\chi(n)\}\wh{R^{\pm}}(\frac{\log n}{2\pi})$ by $\Re\{\chi(n)\wh{R^{\pm}}(\frac{\log n}{2\pi})\}$ on the left-hand side and $\int u\wh{R^{\pm}}(u)^2$ by $\int u|\wh{R^{\pm}}(u)|^2$ on the right-hand side. The estimates in both corollaries remain unchanged after we replace $1/2$ by $1/2+iT_0$ and $|\gamma_{\chi,0}|$ by the distance from $1/2+iT_0$ to the nearest zero. Finally, we can also replace the latter by the distance to the nearest zero above or below. For example, for $T_0=0$, this just says that our existence and proportion results are not weakened if we restrict to low-lying zeros with $\gamma\geq 0$ (or $\leq 0$).

\section{Acknowledgment}
I would like to thank Professor Ghaith Hiary for several interesting discussions. I also thank the anonymous referee for their suggestions.

\printbibliography

\end{document}